\documentclass[letterpaper,11pt]{amsart}
\usepackage{amssymb}
\usepackage{amsthm}
\usepackage[utf8]{inputenc}
\usepackage[margin=1.0in]{geometry}
\usepackage{pst-node}
\usepackage{tikz-cd} 
\usepackage{hyperref}
\newtheorem{theorem}{Theorem}[section]
\newtheorem{lemma}[theorem]{Lemma}
\newtheorem{proposition}[theorem]{Proposition}
\newtheorem{corollary}[theorem]{Corollary}

\usepackage[T1]{fontenc}

\usepackage[english]{babel}
\usepackage[backend=biber,style=alphabetic]{biblatex}
\theoremstyle{definition}
\newtheorem{definition}{Definition}[section]
\theoremstyle{remark}
\newtheorem{remark}{Remark}[section]
\theoremstyle{example}

\usepackage{amsmath}
\usepackage{amsfonts}
\usepackage{tikz-cd}

\def\lim{\mathop{\rm lim}\nolimits}
\def\colim{\mathop{\rm colim}\nolimits}

\def\Spd{\mathop{\rm Spd}}

\def\Bun{\mathrm{Bun}}

\def\Perf{\mathrm{Perf}}

\def\Sht{\mathrm{Sht}}
\def\Rep{\mathrm{Rep}}
\def\D{\mathrm{D}}

\newcommand{\Dlis}{\mathrm{D}_{\mathrm{lis}}}

\newcommand{\ol}{\overline}

\usepackage[english]{babel}
\usepackage[utf8]{inputenc}

\def\ol{\overline}

\def\lim{\mathop{\rm lim}\nolimits}
\def\colim{\mathop{\rm colim}\nolimits}

\def\Spd{\mathop{\rm Spd}}

\def\Bun{\mathrm{Bun}}

\def\Perf{\mathrm{Perf}}

\def\Sht{\mathrm{Sht}}
\def\Rep{\mathrm{Rep}}
\def\D{\mathrm{D}}

\title{Zelevinsky Duality on Basic Local Shimura Varieties}
\author{Linus Hamann}

\begin{document}
\maketitle
\textbf{Abstract.} We give a simple proof of a general result describing the action of the Zelevinsky involution on the cohomology of certain basic local Shimura varieties, using the machinery of Fargues-Scholze \cite{FS}. As an application, we generalize earlier results of Fargues \cite{Fa} and Mieda \cite{M} on the action of the Zelevinsky involution on the cohomology of $\mathrm{GL}_{n}$ and $\mathrm{GSp}_{4}$ type basic local Shimura varieties, respectively.
\tableofcontents
\section*{Acknowledgements} I would like to thank David Hansen for helpful discussions, Peter Scholze and the referee for pointing out some small errors, and the MPIM Bonn for their hospitality during part of the completion of this project. This version corrects some small errors and adds clarifications to the published version identified after A.I. review. 
\section*{Notation}
\begin{enumerate}
    \item Let $\ell \neq p$ be distinct primes.
    \item Let $G/\mathbb{Q}_{p}$ be a quasi-split connected reductive group over the $p$-adic numbers. 
    \item Fix $T \subset B \subset G$ a maximal torus and Borel of $G$.
    \item Let $\rho_{G}$ denote the half sum of all positive roots of $G$.
    \item Let $Z(G)$ denote the center of $G$.
    \item Set $\Breve{\mathbb{Q}}_{p}$ to be the completion of the maximal unramified extension of $\mathbb{Q}_{p}$, with Frobenius $\sigma$.
    \item Set $\mathbb{C}_{p} := \hat{\overline{\mathbb{Q}}}_{p}$ to be the completion of the algebraic closure of $\mathbb{Q}_{p}$.
    \item Let $B(G) = G(\Breve{\mathbb{Q}}_{p})/(g \sim hg\sigma(h)^{-1})$ denote the Kottwitz set of $G$.
    \item For an element $b \in B(G)$, we let $J_{b}$ denote the $\sigma$-centralizer of $b$.
    \item We let $\mathbb{C}_{p}^{\flat} := \lim_{x \mapsto x^{p}} \mathbb{C}_{p}$ be the tilt of $\mathbb{C}_{p}$. We write $X$ for the (adic) Fargues-Fontaine curve defined by $\mathbb{C}_{p}^{\flat}$.
    \item To an element $b \in B(G)$, we let $\mathcal{E}_{b}$ be the associated $G$-bundle on $X$.
    \item Let $\Perf_{\overline{\mathbb{F}}_{p}}$ denote the category of perfectoid spaces in characteristic $p$ over the algebraic closure $\overline{\mathbb{F}}_{p}$ of $\mathbb{F}_{p}$ equipped with the $v$-topology. For $S \in \Perf_{\overline{\mathbb{F}}_{p}}$, we write $X_{S}$ for the relative (adic) Fargues-Fontaine curve over $S$.
    \item We let $\Bun_{G}$ denote the $v$-stack sending $S \in \Perf_{\overline{\mathbb{F}}_{p}}$ to the groupoid of $G$-bundles on $X_{S}$. 
    \item Let $\Dlis(\Bun_{G},\overline{\mathbb{Q}}_{\ell})$ (resp. $\Dlis(\Bun_{G},\overline{\mathbb{Q}}_{\ell})^{\omega}$) be the category of (resp. compact) lisse-\'etale $\overline{\mathbb{Q}}_{\ell}$-sheaves, as defined in \cite[Section~VII.6]{FS}.
    \item For a finite extension $E/\mathbb{Q}_{p}$, set $W_{E}$ to be the Weil group of $E$. We let $\Dlis(\Bun_{G},\overline{\mathbb{Q}}_{\ell})^{BW_{E}}$ be the category of objects in $\Dlis(\Bun_{G},\overline{\mathbb{Q}}_{\ell})$ with continuous $W_{E}$-action, as defined in \cite[Section~IX.1]{FS}. 
    \item We fix a square root of $p$ in $\overline{\mathbb{Q}}_{\ell}$. All half Tate twists will be defined with respect to this choice.
\end{enumerate}
\section{Introduction}
Let $\mu$ be a geometric dominant cocharacter of $G$ with reflex field $E/\mathbb{Q}_{p}$, and set $b \in B(G,\mu)$ to be the unique basic element in the $\mu$-admissible locus (Definition 2.1). The triple $(G,b,\mu)$ defines a diamond \cite{SW}
\[ \Sht(G,b,\mu)_{\infty} \rightarrow \Spd(\Breve{E}) \]
parametrizing modifications $\mathcal{E}_{b} \rightarrow \mathcal{E}_{0}$ with meromorphy bounded by $\mu$ on $X$, where $\Breve{E}$ is the compositum $E\Breve{\mathbb{Q}}_{p}$. The space $\Sht(G,b,\mu)_{\infty}$ carries an action of $G(\mathbb{Q}_{p}) \times J_{b}(\mathbb{Q}_{p})$ and moreover carries a (non-effective) descent datum from $\Breve{E}$ down to $E$. This allows us to consider the tower of quotients
\[ \Sht(G,b,\mu)_{\infty}/\underline{K} =: \Sht(G,b,\mu)_{K} \]
for varying open compact $K \subset G(\mathbb{Q}_{p})$. If $\mu$ is minuscule, we say that the triple $(G,b,\mu)$ defines a basic local Shimura datum. In this case, the members of this tower are represented by rigid analytic spaces and define local Shimura varieties in the sense of Rapoport-Viehmann \cite{RV}. Let $\overline{\mathbb{Q}}_{\ell}$ be the algebraic closure of the $\ell$-adic numbers. The geometric Satake equivalence of Fargues-Scholze defines a $\mathbb{Z}_{\ell}$-sheaf $\mathcal{S}_{\mu}$ \cite[Chapter~VI]{FS} attached to $\mu$ on the spaces $\Sht(G,b,\mu)_{K}$ compatible with the actions of $G(\mathbb{Q}_{p})$ and $J_{b}(\mathbb{Q}_{p})$, and if $\mu$ is minuscule $\mathcal{S}_{\mu}$ is (up to a shift and twist) just the constant sheaf. Letting $\Sht(G,b,\mu)_{K,\mathbb{C}_{p}}$ denote the base-change of these spaces to $\mathbb{C}_{p}$, we can now define the complex
\[ R\Gamma_{c}(G,b,\mu) := \colim_{K \rightarrow \{1\}} R\Gamma_{c}(\Sht(G,b,\mu)_{K,\mathbb{C}_{p}},\mathcal{S}_{\mu}) \otimes \overline{\mathbb{Q}}_{\ell}\]
of $G(\mathbb{Q}_{p}) \times J_{b}(\mathbb{Q}_{p}) \times W_{E}$-modules. It is an important problem to describe the cohomology of these complexes as they are known to realize instances of the local Langlands correspondences and can be related to the cohomology of global Shimura varieties via uniformization in the case that $\mu$ is minuscule. More specifically, for $\pi$ (resp. $\rho$) a smooth irreducible representation of $G(\mathbb{Q}_{p})$ (resp. $J_{b}(\mathbb{Q}_{p})$) one wants to understand the $\pi$ (resp. $\rho$)-isotypic part. In other words, the complexes
\[ R\Gamma_{c}(G,b,\mu)[\pi] := R\Gamma_{c}(G,b,\mu) \otimes^{\mathbb{L}}_{\mathcal{H}(G)} \pi \]
and
\[ R\Gamma_{c}(G,b,\mu)[\rho] := R\Gamma_{c}(G,b,\mu) \otimes^{\mathbb{L}}_{\mathcal{H}(J_{b})} \rho \]
where $\mathcal{H}(G) := C^{\infty}_{c}(G(\mathbb{Q}_{p}),\overline{\mathbb{Q}}_{\ell})$ (resp. $\mathcal{H}(J_{b})$) is the usual smooth Hecke algebra of $G$ (resp. $J_{b}$). It is a recent result of Fargues-Scholze that $R\Gamma_{c}(G,b,\mu)[\pi]$ (resp. $R\Gamma_{c}(G,b,\mu)[\rho]$) are valued in admissible $J_{b}(\mathbb{Q}_{p})$ (resp. $G(\mathbb{Q}_{p})$)-representations of finite length carrying a continuous action of $W_{E}$ \cite[Corollary I.7.3 and Section IX.3]{FS}. 
\\\\
Let $\phantom{}^{L}G$ be the L-group of $G$. In instances where the local Langlands correspondences are known and the $L$-parameter $\phi: W_{E} \times \mathrm{SL}_{2}(\mathbb{C}) \rightarrow \phantom{}^{L}G(\mathbb{C})$ associated to $\pi$ and $\rho$ is supercuspidal (i.e. the restriction to the $\mathrm{SL}_{2}(\mathbb{C})$-factor acts trivially and $\phi$ does not factor through a proper Levi subgroup of $\phantom{}^{L}G$) it is expected that $R\Gamma_{c}(G,b,\mu)[\pi]$ (resp. $R\Gamma_{c}(G,b,\mu)[\rho]$) should be concentrated in middle degree $0$ \cite[Theorem~1.1]{Han} and valued in representations lying in the $L$-packet $\Pi_{\phi}(J_{b})$ (resp. $\Pi_{\phi}(G)$). Even in this case, verifying this expectation is a very difficult problem and forms the content of the Kottwitz conjecture \cite[Conjecture~7.3]{RV}.
\\\\
In this note, we would like to prove a result that concerns the structure of the cohomology of $R\Gamma_{c}(G,b,\mu)[\pi]$ and $R\Gamma_{c}(G,b,\mu)[\rho]$ when $\pi$ or $\rho$ is a supercuspidal, but does not necessarily have supercuspidal $L$-parameter. This is the case in which the $L$-parameter is discrete but the $\mathrm{SL}_{2}$-factor acts non-trivially. This is a much more difficult problem. Even for $G = \mathrm{GL}_{n}$, $J_{b}$ an inner form of $\mathrm{GL}_{n}$, and $\rho = \mathbf{1}$ the trivial representation, the cohomology of the complex $R\Gamma_{c}(G,b,\mu)[\rho]$ is concentrated in multiple different degrees \cite{SS2,Dat}, and seems, in general, to be valued in representations lying in either the $L$-packet $\Pi_{\phi}(G)$ or the $A$-packet $\Pi_{\psi}(G)$ \cite{I,IM}, where $\psi$ is the $A$-parameter:
\[ \psi: W_{\mathbb{Q}_{p}} \times \mathrm{SL}_{2}(\mathbb{C}) \times \mathrm{SL}_{2}(\mathbb{C}) \rightarrow W_{\mathbb{Q}_{p}} \times \mathrm{SL}_{2}(\mathbb{C}) \times \mathrm{SL}_{2}(\mathbb{C}) \xrightarrow{\phi} \phantom{}^{L}G(\mathbb{C}) \]
Here the first map is given by swapping the $\mathrm{SL}_{2}$-factors and the second map is given by trivially extending $\phi$ on the second factor. To aid in understanding the cohomology in this case, one considers the (derived) Zelevinsky involution. 
\begin{definition}
For $G/\mathbb{Q}_{p}$ a connected reductive group, let $\mathcal{D}_{c}(G)$ be the convolution algebra of compactly supported $\overline{\mathbb{Q}}_{\ell}$-valued distributions on $G(\mathbb{Q}_{p})$. It contains the usual smooth Hecke algebra $\mathcal{H}(G)$ as a sub-algebra. We let  $\D(\mathcal{D}_{c}(G),\ol{\mathbb{Q}}_{\ell})$ (resp. $\D(G(\mathbb{Q}_{p}),\ol{\mathbb{Q}}_{\ell})$) denote the unbounded derived category of $\mathcal{D}_{c}(G)$-modules (resp. smooth $G(\mathbb{Q}_{p})$-modules). If we let $e_{K} \in \mathcal{H}(G) \subset \mathcal{D}_{c}(G)$ be the idempotent for the Hecke algebra defined by a compact open subgroup $K \subset G(\mathbb{Q}_{p})$ then we can consider the functor
\[ \infty_{\mathcal{D}}: \D(\mathcal{D}_{c}(G),\overline{\mathbb{Q}}_{\ell}) \rightarrow  \D(G(\mathbb{Q}_{p}),\overline{\mathbb{Q}}_{\ell}) \] 
\[ A \mapsto \colim_{K \rightarrow \{1\}} e_{K}A \]
which is in particular exact and so has no higher derived functors (see \cite[Section~2]{M} and \cite[Section~1]{Fa} for details). Then we define the (derived) Zelevinsky involution to be 
\[ \mathbb{D}_{G}: \D(G(\mathbb{Q}_{p}),\overline{\mathbb{Q}}_{\ell})^{op} \rightarrow \D(G(\mathbb{Q}_{p}),\overline{\mathbb{Q}}_{\ell}) \]
\[ A \mapsto \infty_{\mathcal{D}} \circ \mathcal{R}Hom_{\mathcal{H}(G)}(A,\mathcal{H}(G)) \]
where $\mathcal{H}(G)$ is viewed as a Hecke bi-module. 
\end{definition}
On the subcategory of compact objects, this defines a contravariant equivalence and has the property that $\mathbb{D}_{G}^{2} \simeq Id$. In this note, we seek to describe the complexes $\mathbb{D}_{G}(R\Gamma_{c}(G,b,\mu)[\rho])$ and $\mathbb{D}_{J_{b}}(R\Gamma_{c}(G,b,\mu)[\pi])$. This is desirable since (up to contragredients and shifts) $\mathbb{D}_{G}$ should exchange the $L$-packet $\Pi_{\phi}(G)$ and the $A$-packet $\Pi_{\psi}(G)$ \cite{Hir}, so better understanding the structure of this complex should give insight into the complicated way these different representations appear and interact with each other in the cohomology of $R\Gamma_{c}(G,b,\mu)[\rho]$. As a first step towards this goal, we would like to prove the existence of isomorphisms
\[ \mathbb{D}_{J_{b}}(R\Gamma_{c}(G,b,\mu)[\pi]) \simeq R\Gamma_{c}(G,b,\mu)[\mathbb{D}_{G}(\pi)]\]
and
\[ \mathbb{D}_{G}(R\Gamma_{c}(G,b,\mu)[\rho]) \simeq R\Gamma_{c}(G,b,\mu)[\mathbb{D}_{J_{b}}(\rho)]\]
as complexes of $J_{b}(\mathbb{Q}_{p}) \times W_{E}$ and $G(\mathbb{Q}_{p}) \times W_{E}$-modules, respectively.
\\\\
In general, this shouldn't be true, but, under a Hodge-Newton reducibility condition, we show that it is. In particular, our main theorem is as follows.
\begin{theorem}
For a basic local Shimura datum $(G,b,\mu)$ such that the set $B(G,\mu)$ is Hodge-Newton reducible (Definition 2.2) and $\pi$ (resp. $\rho$) a supercuspidal representation of $G(\mathbb{Q}_{p})$ (resp. $J_{b}(\mathbb{Q}_{p})$), we have  isomorphisms
\[ \mathbb{D}_{J_{b}}(R\Gamma_{c}(G,b,\mu)[\pi]) \simeq R\Gamma_{c}(G,b,\mu)[\mathbb{D}_{G}(\pi)] \]
and 
\[ \mathbb{D}_{G}(R\Gamma_{c}(G,b,\mu)[\rho]) \simeq R\Gamma_{c}(G,b,\mu)[\mathbb{D}_{J_{b}}(\rho)] \]
as complexes of $J_{b}(\mathbb{Q}_{p}) \times W_{E}$ and $G(\mathbb{Q}_{p}) \times W_{E}$-modules, respectively. Here $\sigma \in W_E$ acts on $\mathbb{D}_{J_b}(-)$ (resp.\ $\mathbb{D}_G(-)$) via $\mathbb{D}_{J_b}$ (resp.\ $\mathbb{D}_G$) acts via contragredients applied to the original action. 
\end{theorem}
\begin{remark}
A similar result should be provable through our methods with integral or torsion coefficients, for $\rho$ or $\pi$ any supercuspidal $\ell$-complete or $\ell$-modular representation, respectively. We note that in the above we have assumed $\mu$ is minuscule; however, we conjecture that this is true even in the non-minuscule case.
\end{remark}
The condition that the set $B(G,\mu)$ is Hodge-Newton reducible roughly says that, for every non-basic element $b' \in B(G,\mu)$, the Newton polygon of $b'$ seen as an element of a positive Weyl chamber touches the Hodge polygon defined by $\mu$ away from the endpoints. We note that there is a complete group-theoretic classification of pairs $(G,\mu)$ such that $B(G,\mu)$ is Hodge-Newton reducible \cite[Theorems~3.3,3.5]{GHN}.
\\\\
We explain the idea of proof for the $\pi$-isotypic part with the proof for the $\rho$-isotypic part being analogous. The key insight is to consider the geometric analogue of the Zelevinsky involution acting on sheaves on the moduli stack $\Bun_{G}$
\[ \mathbb{D}_{BZ}: \Dlis(\Bun_{G},\overline{\mathbb{Q}}_{\ell})^{\omega,op} \rightarrow  \Dlis(\Bun_{G},\overline{\mathbb{Q}}_{\ell})^{\omega} \]
as defined in \cite[Section~V.5,Section~VII.7]{FS}. It defines a contravariant equivalence of $\Dlis(\Bun_{G},\overline{\mathbb{Q}}_{\ell})^{\omega}$ and has the property that $\mathbb{D}_{BZ}^{2}$ is naturally isomorphic to the identity. Moreover, when restricted to the subcategories $\D(J_{b}(\mathbb{Q}_{p}),\overline{\mathbb{Q}}_{\ell})^{\omega}$ coming from basic Harder-Narasimhan(=HN)-strata of $\Bun_{G}$, it is given by $\mathbb{D}_{J_{b}}$. Now, with this in hand, it is useful to reinterpret the complex $R\Gamma_{c}(G,b,\mu)[\pi]$ as 
\[ j_{b}^{*}T_{\mu^{-1}}j_{1!}(\mathcal{F}_{\pi}) \simeq R\Gamma_{c}(G,b,\mu)[\pi] \]
where $j_{1}$ and $j_{b}$ are the inclusions of the HN-strata of $\Bun_{G}$ corresponding to the trivial element $1 \in B(G)$ and $b \in B(G)$, respectively, $\mathcal{F}_{\pi} \in \D(G(\mathbb{Q}_{p}),\overline{\mathbb{Q}}_{\ell})$ is the sheaf on the basic HN-stratum $\Bun_{G}^{1} \simeq [\ast/\underline{G(\mathbb{Q}_{p})}]$ defined by $\pi$, and $T_{\mu^{-1}}$ is the Hecke operator associated to $\mu^{-1}$, a dominant inverse of $\mu$ acting on $\Dlis(\Bun_{G},\overline{\mathbb{Q}}_{\ell})$. Given this isomorphism, Theorem 1.1 would now follow from knowing that $\mathbb{D}_{BZ}$ commutes with all operations applied to $\mathcal{F}_{\pi}$ on the LHS. For $j_{1!}$ and the Hecke operator $T_{\mu^{-1}}$, this is shown by Fargues-Scholze. The key point is to show that it commutes with the restriction functor $j_{b}^{*}$. In general, this is not true if the complex $T_{\mu^{-1}}j_{1!}(\mathcal{F}_{\pi})$ is supported on non-basic HN-strata, but in the case that $B(G,\mu)$ is Hodge-Newton reducible and $\pi$ is supercuspidal we show that this cannot occur. The key point is that the restriction of $T_{\mu^{-1}}j_{1!}(\mathcal{F}_{\pi})$ to any non-basic HN-stratum can be described in terms of the $\pi$-isotypic part of $R\Gamma_{c}(G,b',\mu)$ for non-basic $b' \in B(G,\mu)$. However, by Hodge-Newton reducibility, the space $\Sht(G,b',\mu)_{\infty}$ is parabolically induced as a space with $G(\mathbb{Q}_{p})$-action from a proper parabolic of $G$, realizing some form of the Harris-Viehmann conjecture \cite[Conjecture~8.5]{RV}. Therefore, since $\pi$ is supercuspidal, it cannot contribute to the cohomology of these spaces, so the restriction vanishes. 
\\\\
With Theorem 1.1 in hand, we can start to derive some useful cohomological consequences. Let $\chi$ denote the central character of our fixed supercuspidal $\pi$. Then the complex $R\Gamma_{c}(G,b,\mu)[\pi]$ has cohomology valued in $J_{b}(\mathbb{Q}_{p})$-representations with central character equal to $\chi$, where we recall that $J_{b}$ is an extended pure inner form of $G$ and therefore we have an isomorphism: $Z(J_{b}) \simeq Z(G)$. Let $I_{\chi}(J_{b})$ denote the set of equivalence classes of inertial cuspidal data $\mathfrak{s} = (M,\sigma)$, where $M \subset J_{b}$ is a Levi subgroup of $J_{b}$ and $\sigma$ is a supercuspidal representation with $\sigma|_{Z(G)(\mathbb{Q}_{p})} = \chi$. We can consider the Bernstein decomposition 
\[ R\Gamma_{c}(G,b,\mu)[\pi] \simeq \bigoplus_{\mathfrak{s} \in I_{\chi}(J_{b})}  R\Gamma_{c}(G,b,\mu)[\pi]_{\mathfrak{s}} \]
where the cohomology of $R\Gamma_{c}(G,b,\mu)[\pi]_{\mathfrak{s}}$ takes values in representations with supercuspidal support specified by $\mathfrak{s}$. On the other hand, since $\pi$ is supercuspidal, $R\Gamma_{c}(G,b,\mu)[\mathbb{D}_{G}(\pi)]$ is, up to a shift, isomorphic to $R\Gamma_{c}(G,b,\mu)[\pi^{*}]$, where $\pi^{*}$ is the contragredient of $\pi$. Therefore, this complex takes values in representations with central character equal to $\chi^{-1}$. Hence, we again obtain a Bernstein decomposition:
\[ R\Gamma_{c}(G,b,\mu)[\pi^{*}] \simeq \bigoplus_{\mathfrak{s} \in I_{\chi^{-1}}(J_{b})} R\Gamma_{c}(G,b,\mu)[\pi^{*}]_{\mathfrak{s}} \]
Fixing now $\mathfrak{s} = (M,\sigma) \in I_{\chi}(J_{b})$, we let $\mathfrak{s}^{*} = (M,\sigma^{*}) \in I_{\chi^{-1}}(J_{b})$. Using Theorem 1.1, one can show that one has (up to a shift) an isomorphism
\[ \mathbb{D}_{J_{b}}(R\Gamma_{c}(G,b,\mu)[\pi]_{\mathfrak{s}}) \simeq R\Gamma_{c}(G,b,\mu)[\pi^{*}]_{\mathfrak{s}^{*}} \]
of complexes of $J_{b}(\mathbb{Q}_{p}) \times W_{E}$-modules. 
To a given $\mathfrak{s}$, we let $i_{J_{b}}(\mathfrak{s}) := r_{J_{b}} - r_{M}$, where $r_{J_{b}}$ (resp. $r_{M}$) is the split semisimple rank of $J_{b}$ (resp. $M$). If one has an irreducible smooth representation $\rho$, then the cohomology of $\mathbb{D}_{J_{b}}(\rho)$ sits in a single degree determined by $i_{J_{b}}(\mathfrak{s})$ for $\rho$ with supercuspidal support $\mathfrak{s}$. We can define $\mathrm{Zel}(\rho)$ (Definition 3.1) to be the representation determined by the non-zero cohomology. As a consequence of this, we can deduce the following result from Theorem 1.1 by passing to cohomology.
\begin{theorem}
For a basic local Shimura datum $(G,b,\mu)$ such that $B(G,\mu)$ is Hodge-Newton reducible, let $\pi$ be a supercuspidal representation of $G(\mathbb{Q}_{p})$ with central character $\chi$, $\mathfrak{s} \in I_{\chi}(J_{b})$ with associated integer $i_{J_{b}}(\mathfrak{s})$ as above, $\rho$ a smooth irreducible representation with supercuspidal support given by $\mathfrak{s}$, and $\sigma$ an irreducible representation of $W_{E}$. Then $\rho\boxtimes\sigma$ occurs as a subquotient of $H^q(R\Gamma_c(G,b,\mu)[\pi])$ if and only if
$\mathrm{Zel}(\rho)\boxtimes\sigma^*$ occurs as a subquotient of
$H^{-q+i_{J_b}(\mathfrak{s})}(R\Gamma_c(G,b,\mu)[\pi^*])$. The analogous statement is true with the roles
of $\pi$ and $\rho$ interchanged (with $\mathfrak{s}\in I_\chi(G)$ and $i_G(\mathfrak{s})$).
\end{theorem}
\begin{remark}
\begin{enumerate}
    \item Using this result, one should be able to make some more precise claims about the cohomological degrees that the complex $R\Gamma_{c}(G,b,\mu)[\pi]$ (resp. $R\Gamma_{c}(G,b,\mu)[\rho]$) sits in for $\pi$ (resp. $\rho$) supercuspidal. For example, conjecturally the spaces $\Sht(G,b,\mu)_{K}$ should be Stein (see \cite[Section~1.3]{Han} for a discussion of this) in the sense that the complex $R\Gamma_{c}(\Sht(G,b,\mu)_{K},\mathcal{S}_{\mu})$ should be concentrated in degrees $0 \leq i \leq d$.\footnote{We note that in the minuscule case $\mathcal{S}_{\mu} \simeq \mathbb{Z}_{\ell}[d](\frac{d}{2})$.} As a consequence, since $\pi$ (resp. $\rho$) supercuspidal is injective/projective in the category of representations with fixed central character (though one needs to also account for the fact that $R\Gamma_{c}(G,b,\mu)$ does not have fixed central character (see the proof of \cite[Corollaries~5.11 and 5.12]{M})), we know that $R\Gamma_{c}(G,b,\mu)[\pi]$ (resp. $R\Gamma_{c}(G,b,\mu)[\rho]$) is concentrated in degrees $0 \leq i \leq d$, but then, by Theorem 1.2, this is enough to conclude that $R\Gamma_{c}(G,b,\mu)[\pi]_{\mathfrak{s}}$ (resp. $R\Gamma_{c}(G,b,\mu)[\rho]_{\mathfrak{s}}$) can only be concentrated in degrees $0 \leq i \leq i_{J_{b}}(\mathfrak{s})$. 
    \item In the case that $L/\mathbb{Q}_{p}$ is a finite extension, $G = \mathrm{Res}_{L/\mathbb{Q}_{p}}\mathrm{GL}_{n}$ is the restriction of scalars of $\mathrm{GL}_{n}$, and $\mu$ the Lubin-Tate cocharacter, this is a result of Fargues \cite[Theorem~4.6]{Fa}. Similarly, for $G = \mathrm{GSp}_{4}/\mathbb{Q}_{p}$, and $\mu$ the Siegel cocharacter, this is  a result of Mieda \cite[Corollary~1.2 (ii)]{M}. 
    \item Unlike the arguments of Fargues and Mieda, we completely avoid considering integral models of the local Shimura varieties, working solely on the generic fiber. This substantially simplifies the proof, allowing one to completely avoid using nearby cycles and studying the special fiber, albeit at the cost of using much more sophisticated geometric tools. 
    \item In the case that $G = \mathrm{GSp}_{4}/\mathbb{Q}_{p}$, $\mu = (1,1,0,0)$, $\phi$ is a discrete (not necessarily supercuspidal) $L$-parameter with associated $A$-parameter $\psi$ as above, and $\rho$ a (not necessarily supercuspidal) representation in the $L$-packet or $A$-packet over $\phi$ or $\psi$, respectively, Ito and Mieda \cite{IM} have leveraged this result in conjunction with basic uniformization and global results to fully understand the cohomology of the summand $R\Gamma_{c}(G,b,\mu)[\rho]_{sc}$ of $R\Gamma_{c}(G,b,\mu)[\rho]$, where $G(\mathbb{Q}_{p})$ acts via a supercuspidal representation, in terms of members of the $L$-packet and $A$-packet $\Pi_{\phi}(G)$ and $\Pi_{\psi}(G)$. It would be interesting to carry out their analysis in other Hodge-Newton reducible situations. For example, in the case where $G = \mathrm{GU}(1,n - 1)/\mathbb{Q}_{p}$ is an unramified unitary group and $\mu = (1,0,\ldots,0,0)$, the form of the cohomology in the case that $n = 3$ is partially described by \cite{I}. 
    \end{enumerate}

\end{remark}
In section 2, we discuss Hodge-Newton reducibility and prove Theorem 1.1. In section 3, we discuss the Zelevinsky involution in greater detail and show how to deduce Theorem 1.2 from Theorem 1.1. 
\section{Proof of the Main Theorem}
In this section, we prove Theorem 1.1. We first address the notion of Hodge-Newton reducibility and derive a key geometric consequence. We recall that the Kottwitz set $B(G)$ is equipped with two maps.
\begin{itemize}
    \item The slope homomorphism
    \[ \nu: B(G) \rightarrow X_*(T_{\overline{\mathbb{Q}}_{p}})^{+,\Gamma}_{\mathbb{Q}} \]
    \[ b \mapsto \nu_{b}, \]
    where $\Gamma := \mathrm{Gal}(\overline{\mathbb{Q}}_{p}/\mathbb{Q}_{p})$ and 
    $X_*(T_{\overline{\mathbb{Q}}_{p}})_{\mathbb{Q}}^{+}$ is the set of rational dominant cocharacters of $G$. 
    \item The Kottwitz invariant
    \[ \kappa: B(G) \rightarrow \pi_{1}(G)_{\Gamma}, \]
    where $\pi_1(G)$ is Borovoi's fundamental group.
\end{itemize}
Now, given a geometric dominant cocharacter $\mu$ of $G$ with Galois extension $E/\mathbb{Q}_{p}$ stabilizing $\mu$, we can define the element:
\[ \tilde{\mu} := \frac{1}{[E:\mathbb{Q}_{p}]} \sum_{\gamma \in \mathrm{Gal}(E/\mathbb{Q}_{p})} \gamma(\mu) \in X_*(T_{\overline{\mathbb{Q}}_{p}})^{+,\Gamma}_{\mathbb{Q}} \]
We let $\mu^{\flat}$ be the image of $\mu$ in $\pi_1(G)_{\Gamma}$. This allows us to make the following definition. 
\begin{definition}
Let $b$ be a basic element in the Kottwitz set $B(G)$ of $G$. We define the set $B(G,\mu,b)$ to be the set of $b' \in B(G)$ for which $\nu_{b'} - \nu_{b} \leq \tilde{\mu}$ with respect to the dominance ordering and $\kappa(b') - \kappa(b) = \mu^{\flat}$. If $b$ is the trivial element we will denote this by $B(G,\mu)$. 
\end{definition}
With this in hand, we can define the relevant notion of Hodge-Newton reducibility. 
\begin{definition}
For $b \in B(G)$ basic, we say that $B(G,\mu,b)$ is Hodge-Newton reducible if, for all non-basic $b' \in B(G,\mu,b)$, there exists a proper standard Levi subgroup $M \subset G$ containing the centralizer of $\nu_{b'}$ and $b_{M}, b'_{M} \in B(M)$ such that:
\begin{itemize}
    \item $b$ and $b'$ are in the images of $b_{M}$ and $b'_{M}$ under the natural map $B(M) \rightarrow B(G)$, respectively, 
    \item $\mu$ is viewed as a cocharacter of $M$ via its $B$-dominant representative and $b_{M}' \in B(M,\mu,b_{M})$, and $\nu_{b_{M}'}$ is dominant with respect to the fixed choice of Borel.
\end{itemize}
\end{definition}
\begin{remark}
We note that, since we assumed $G$ is quasi-split, by \cite[Lemma~4.11]{CFS}, $B(G,\mu)$ being Hodge-Newton reducible is equivalent to saying that $B(G,\mu)$ is fully Hodge-Newton decomposable in the sense of \cite[Definition~4.10]{CFS}. If $b \in B(G,\mu)$ denotes the unique basic element and $\mu^{-1}$ is a dominant inverse of $\mu$ then it follows that, if $B(G,\mu)$ is Hodge-Newton reducible, then $B(G,\mu^{-1},b)$ is also Hodge-Newton reducible in our sense by \cite[Corollary~4.16 and Lemma~4.13]{CFS} and \cite[Lemma~4.11]{CFS} (see the first part of the proof of \cite[Theorem~6.1]{CFS} for details). 
\end{remark}
We now seek to derive a key geometric consequence of this condition. Fix an element $b \in B(G)$ and let $b' \in B(G,\mu,b)$. We define the Shtuka space
\[ \Sht(G,b,b',\mu)_{\infty} \rightarrow \Spd(\Breve{E}) \]
to be the space parametrizing modifications 
\[ \mathcal{E}_{b'} \dashrightarrow \mathcal{E}_{b} \]
of $G$-bundles on the Fargues-Fontaine curve $X$ with meromorphy bounded by $\mu$. This has commuting actions of $J_{b}(\mathbb{Q}_{p})$ and $J_{b'}(\mathbb{Q}_{p})$ coming from automorphisms of $\mathcal{E}_{b}$ and $\mathcal{E}_{b'}$, respectively. We define as before the tower 
\[ \Sht(G,b,b',\mu)_{K} := \Sht(G,b,b',\mu)_{\infty}/\underline{K} \]
for varying open compact subgroups $K \subset J_{b}(\mathbb{Q}_{p})$. Assume that
$b$ is basic from now on. Then $J_{b}$ is an extended pure inner form of $G$. It follows that one has an isomorphism $\Bun_{G} \simeq \Bun_{J_{b}}$  by mapping a $G$-bundle $\mathcal{E}$ on $X$ to the $J_{b}$-torsor of isomorphisms $\mathcal{I}som(\mathcal{E}_{b},\mathcal{E})$. This implies that the space of Shtukas $\Sht(G,b,b',\mu)_{\infty}$ is isomorphic to the space of Shtukas $\Sht(J_{b},b'b^{-1},\mu)_{\infty}$ for the group $J_{b}$ (cf. \cite[Section~5.1]{CFS}). This isomorphism means that the cocharacter $\mu$ and the Geometric Satake correspondence for $J_{b}$ \cite[Chapter~VI]{FS} will give rise to a $\mathbb{Z}_{\ell}$-sheaf $\mathcal{S}_{\mu}$ on the space $\Sht(J_{b},b'b^{-1},\mu)_{\infty}$ and in turn on $\Sht(G,b,b',\mu)_{\infty}$ compatible with the actions of $J_{b}(\mathbb{Q}_{p})$ and $J_{b'}(\mathbb{Q}_{p})$. This allows us to define the complex
\[ R\Gamma_{c}(G,b,b',\mu) := \colim_{K \rightarrow \{1\}} R\Gamma_{c}(\Sht(G,b,b',\mu)_{K,\mathbb{C}_{p}},\mathcal{S}_{\mu})\otimes \overline{\mathbb{Q}}_{\ell} \]
of $J_{b}(\mathbb{Q}_{p}) \times J_{b'}(\mathbb{Q}_{p}) \times W_{E}$-modules, where $\Sht(G,b,b',\mu)_{K,\mathbb{C}_{p}}$ is the base change of $\Sht(G,b,b',\mu)_{K}$ to $\mathbb{C}_{p}$. Here $E$ is the reflex field of $\mu$. If $b$ is the trivial element then we omit it from the notation. We have the following key result concerning this complex.
\begin{proposition} Suppose that $\mu$ is minuscule and $B(G,\mu,b)$ is Hodge-Newton reducible and let $b' \in B(G,\mu,b)$ be a non-basic element, then the cohomology groups of the complex $R\Gamma_{c}(G,b,b',\mu)$ are parabolically induced as a $J_{b}(\mathbb{Q}_{p})$-modules. 
\end{proposition}
\begin{proof}
In the case that $\mu$ is minuscule, we have an isomorphism: $\mathcal{S}_{\mu} \simeq \mathbb{Z}_{\ell}[d](\frac{d}{2})$. It follows from \cite[Theorem~4.26]{IG} that the cohomology of $R\Gamma_{c}(G,b,b',\mu)$ is parabolically induced as a $J_{b}(\mathbb{Q}_{p})$-module. 
\end{proof}
To $\mu$ and its dominant inverse $\mu^{-1}$, we consider the Hecke operators
\[ T_{\mu}, T_{\mu^{-1}}: \Dlis(\Bun_{G},\overline{\mathbb{Q}}_{\ell}) \rightarrow \Dlis(\Bun_{G},\overline{\mathbb{Q}}_{\ell})^{BW_{E}} \]
defined by mutually dual highest weight representations of highest weight $\mu$ and $\mu^{-1}$ of the Langlands dual group, as defined in \cite[Section~IX.2]{FS}. We can deduce the following consequence for the support of these Hecke operators.
\begin{corollary}
For a basic local Shimura datum $(G,b,\mu)$ such that $B(G,\mu)$ is Hodge-Newton reducible and $\pi$ (resp. $\rho$) a supercuspidal representation of  $G(\mathbb{Q}_{p})$ (resp. $J_{b}(\mathbb{Q}_{p})$), the objects
\[ T_{\mu^{-1}}j_{1!}(\mathcal{F}_{\pi}) \]
and
\[ T_{\mu}j_{b!}(\mathcal{F}_{\rho}) \]
in $\Dlis(\Bun_{G},\overline{\mathbb{Q}}_{\ell})^{BW_{E}}$ only have non-zero restriction to the open HN-stratum $\Bun_{G}^{b}$  and $\Bun_{G}^{1}$ defined by $b$ and the trivial element $1$, respectively. Here $j_{b}: \Bun_{G}^{b} \hookrightarrow \Bun_{G}$ and $j_{1}: \Bun_{G}^{1} \hookrightarrow \Bun_{G}$ are open immersions of the HN-strata defined by $b$ and $1$, respectively, and $\mathcal{F}_{\pi}$ (resp. $\mathcal{F}_{\rho}$) are the lisse-\'etale $\overline{\mathbb{Q}}_{\ell}$-sheaves defined by $\pi$ (resp. $\rho$) on $\Bun_{G}^{1}$ (resp. $\Bun_{G}^{b}$).
\end{corollary}
\begin{proof}
We start with the statement for $\pi$. It follows by \cite[Proposition~A.9]{R} that the locus where $T_{\mu^{-1}}j_{1!}(\mathcal{F}_{\pi})$ is supported is contained in $B(G,\mu) \subset B(G) \simeq |\Bun_{G}|$, where $|\Bun_{G}|$ is the underlying topological space of $\Bun_{G}$. For any $b ' \in B(G)$, it follows by the analysis in \cite[Section~IX.3]{FS} that we have an isomorphism  
\[ j_{b'}^{*}T_{\mu^{-1}}j_{1!}(\mathcal{F}_{\pi}) \simeq R\Gamma_{c}(G,b',\mu)[\pi] \]
as complexes of $J_{b'}(\mathbb{Q}_{p}) \times W_{E}$-modules. 
However, by Proposition 2.1, we know that the cohomology sheaves of the complex $R\Gamma_{c}(G,b',\mu)$ are parabolically induced as a $G(\mathbb{Q}_{p})$-module for non-basic $b'$. Therefore, since $\pi$ is supercuspidal, it follows that $R\Gamma_{c}(G,b',\mu)[\pi] := R\Gamma_{c}(G,b',\mu) \otimes^{\mathbb{L}}_{\mathcal{H}(G)} \pi$ must vanish, using second adjointness and projective resolutions (cf. the proof of \cite[Theorem~4.30]{IG}). For $\rho$, since $b$ is basic, the sheaf
\[ T_{\mu}j_{b!}(\mathcal{F}_{\rho}) \]
will be supported on $B(G,\mu^{-1},b)$ by \cite[Corollary~A.10]{R}. It then follows by \cite[Corollary~4.15]{CFS} that this is also Hodge-Newton reducible, as in Remark 2.1. The restrictions of $T_{\mu}j_{b!}(\mathcal{F}_{\rho})$ to all non-basic HN-strata $b' \in B(G,\mu^{-1},b)$ are isomorphic to $R\Gamma_{c}(G,b,b',\mu^{-1}) \otimes^{\mathbb{L}}_{\mathcal{H}(J_{b})} \rho$ (up to shifts and twists which we suppress). However, by Proposition 2.1, we know that the cohomology of $R\Gamma_{c}(G,b,b',\mu^{-1})$ is parabolically induced as a $J_{b}(\mathbb{Q}_{p})$-module and therefore the restriction vanishes as above, since $\rho$ is supercuspidal. 
\end{proof}
Using this, we now deduce Theorem 1.1. 
\begin{theorem}
For a basic local Shimura datum $(G,b,\mu)$ such that $B(G,\mu)$ is Hodge-Newton reducible and $\pi$ (resp. $\rho$) a supercuspidal representation of $G(\mathbb{Q}_{p})$ (resp. $J_{b}(\mathbb{Q}_{p})$), we have  isomorphisms
\[ \mathbb{D}_{J_{b}}(R\Gamma_{c}(G,b,\mu)[\pi]) \simeq R\Gamma_{c}(G,b,\mu)[\mathbb{D}_{G}(\pi)] \]
and 
\[ \mathbb{D}_{G}(R\Gamma_{c}(G,b,\mu)[\rho]) \simeq R\Gamma_{c}(G,b,\mu)[\mathbb{D}_{J_{b}}(\rho)] \]
as complexes of $J_{b}(\mathbb{Q}_{p}) \times W_{E}$ and $G(\mathbb{Q}_{p}) \times W_{E}$-modules, respectively. Here $\sigma \in W_E$ acts on $\mathbb{D}_{J_b}(-)$ (resp.\ $\mathbb{D}_G(-)$) via $\mathbb{D}_{J_b}$ (resp.\ $\mathbb{D}_G$) by contragredients applied to the original action of $W_{E}$. 
\end{theorem}
\begin{proof}
We explain the proof for the $\pi$-isotypic part, with the proof for the $\rho$-isotypic part being analogous. As in the previous proof, we start with the isomorphism 
\[ j_{b}^{*}T_{\mu^{-1}}j_{1!}(\mathcal{F}_{\pi}) \simeq R\Gamma_{c}(G,b,\mu)[\pi] \]
of complexes of $J_{b}(\mathbb{Q}_{p}) \times W_{E}$-modules. Acting by $\mathbb{D}_{J_{b}}$, we have an isomorphism:  
\[ \mathbb{D}_{J_{b}}j_{b}^{*}T_{\mu^{-1}}j_{1!}(\mathcal{F}_{\pi}) \simeq \mathbb{D}_{J_{b}}R\Gamma_{c}(G,b,\mu)[\pi] \]
Let $\mathbb{D}_{BZ}: \Dlis(\Bun_{G},\overline{\mathbb{Q}}_{\ell})^{\omega,op} \rightarrow \Dlis(\Bun_{G},\overline{\mathbb{Q}}_{\ell})^{\omega}$ be the geometric incarnation of the Zelevinsky involution, as defined in \cite[Section~V.5,Section~VII.7]{FS}. By Corollary 2.2 and \cite[Proposition~VII.7.6]{FS}, it follows that the LHS of the previous equation is isomorphic to:
\[ j_{b}^{*}\mathbb{D}_{BZ}T_{\mu^{-1}}j_{1!}(\mathcal{F}_{\pi}) \]
Now, by \cite[Theorem~I.7.2]{FS} and \cite[Proposition~VII.7.6]{FS}, $\mathbb{D}_{BZ}$ also commutes with the Hecke operator $T_{\mu^{-1}}$ and acts on $j_{1!}(\mathcal{F}_{\pi})$ via $\mathbb{D}_{G}$, respectively; therefore, we can rewrite the LHS as
\[ j_{b}^{*}T_{\mu^{-1}}j_{1!}(\mathcal{F}_{\mathbb{D}_{G}(\pi)}) \]
which, as before, is isomorphic to
\[ R\Gamma_{c}(G,b,\mu)[\mathbb{D}_{G}(\pi)] \]
as a complex of $J_{b}(\mathbb{Q}_{p}) \times W_{E}$-modules. All in all, we obtain the desired identification. 
\end{proof}
\section{Applications}
We start by recalling some facts about the Zelevinsky involution. Set $r$ to be the split rank of the center of $G$. We let $\chi: Z(G)(\mathbb{Q}_{p}) \rightarrow \overline{\mathbb{Q}}_{\ell}^{*}$ be a central character and $I_{\chi}(G)$ be the set of equivalence classes of inertial types $(M,\sigma)$, where $M$ is a Levi subgroup of $G$ and $\sigma$ is a supercuspidal representation of $M$ with $\sigma|_{Z(G)(\mathbb{Q}_{p})} \simeq \chi$. We write $\Rep_{\chi}(G)$ for the category of finite-length smooth representations with central character equal to $\chi$. We consider the Bernstein-decomposition \cite[Theorem~VI.7.2]{Ren}:
\[ \Rep_{\chi}(G) = \bigoplus_{\mathfrak{s} \in I_{\chi}(G)} \Rep_{\mathfrak{s}}(G) \]
For $\mathfrak{s} \in I_{\chi}(G)$,  we let $i_{G}(\mathfrak{s}) = r_{G} - r_{M}$ be the integer defined by the difference of the split semi-simple ranks of $G$ and $M$. We supress $G$ when it is clear from context. We then have the following key result of Schneider-Stuhler. 
\begin{proposition}{\cite[Theorem~III.3.1]{SS1} (cf. \cite[Remark~1.5]{Fa})}
For $\mathfrak{s} \in I_{\chi}(G)$ and $\pi \in \Rep_{\mathfrak{s}}(G)$ viewed as an element of $\D(G(\mathbb{Q}_{p}),\overline{\mathbb{Q}}_{\ell})$ sitting in degree $0$, the cohomology of the complex $\mathbb{D}_{G}(\pi)$ is concentrated only in degrees $i(\mathfrak{s}) + r$. Moreover, if $\mathfrak{s}^{*} := (M,\sigma^{*}) \in I_{\chi^{-1}}(G)$ then it takes values in the Bernstein component $\Rep_{\mathfrak{s}^{*}}(G)$ of $\Rep_{\chi^{-1}}(G)$. 
\end{proposition} 
This gives us the following key definition.
\begin{definition}
For $\pi \in \Rep_{\mathfrak{s}}(G)$ as above, we define $\mathrm{Zel}(\pi) \in \Rep_{\mathfrak{s}^{*}}(G)$ to be the $i(\mathfrak{s}) + r$-cohomology of $\mathbb{D}_{G}(\pi)$. When pre-composed with the contragredient map, $\mathrm{Zel}(-)$ defines an exact categorical equivalence of $\Rep_{\mathfrak{s}}(G)$. In particular, it takes irreducible objects to irreducible objects. 
\end{definition}
One can verify Proposition 3.1 by computing $\mathbb{D}_{G}$ explicitly in terms of the following Deligne-Lusztig complex
\[ \mathrm{DL}(\pi) := 0 \rightarrow \pi \rightarrow \bigoplus_{P} i_{P}^{G}r_{P}^{G}(\pi) \rightarrow \ldots \rightarrow  i_{B}^{G}r_{B}^{G}(\pi) \rightarrow 0   \]
where $i_{P}^{G}$ and $r_{P}^{G}$ are the parabolic induction and Jacquet functors associated to a parabolic subgroup $P$, and the $i$th term of the complex runs over conjugacy classes of parabolics of corank $i$ (in general, if $G$ is not quasi-split this has a similar description ending in a minimal parabolic). The differentials of the complex are an appropriate alternating sum of adjunction maps (via transitivity of parabolic induction). 
If one has a complex $A \in \D(G(\mathbb{Q}_{p}),\overline{\mathbb{Q}}_{\ell})$ then one gets a bi-complex $\mathrm{DL}(A)$, and thereby a functor $\mathrm{DL}: \D(G(\mathbb{Q}_{p}),\overline{\mathbb{Q}}_{\ell}) \rightarrow \D(G(\mathbb{Q}_{p}),\overline{\mathbb{Q}}_{\ell})$. We have the following theorem. 
\begin{theorem}{\cite{SS1,B,BBK}}
For $A \in \D(G(\mathbb{Q}_{p}),\overline{\mathbb{Q}}_{\ell})$ with admissible bounded cohomology we have a quasi-isomorphism
\[ \mathrm{DL}(A^{*}) \simeq \mathbb{D}_{G}(A)[r] \]
where $A^{*}$ denotes the smooth contragredient. 
\end{theorem}
For example, using this, we can immediately deduce the following lemma.
\begin{lemma}
Let $\pi$ be a supercuspidal representation of $G(\mathbb{Q}_{p})$ then we have an isomorphism
\[ \mathbb{D}_{G}(\pi) \simeq \pi^{*}[-r] \]
in $\mathrm{D}(G(\mathbb{Q}_{p}),\ol{\mathbb{Q}}_{\ell})$.
\end{lemma}
Now let's consider the isomorphism supplied by Theorem 1.1
\begin{equation}
 \mathbb{D}_{J_{b}}(R\Gamma_{c}(G,b,\mu)[\pi]) \simeq R\Gamma_{c}(G,b,\mu)[\mathbb{D}_{G}(\pi)] \simeq R\Gamma_{c}(G,b,\mu)[\pi^{*}][-r] 
\end{equation}
where we have used Lemma 3.3. Now if $\chi$ denotes the central character of $\pi$ it follows that all smooth irreducible $J_{b}(\mathbb{Q}_{p})$-representations occurring in the complex $R\Gamma_{c}(G,b,\mu)[\pi]$ (resp. $R\Gamma_{c}(G,b,\mu)[\pi^{*}]$) have central character $\chi$ (resp. $\chi^{-1}$). This follows since the diagonally embedded center of $Z(J_{b})(\mathbb{Q}_{p}) \simeq Z(G)(\mathbb{Q}_{p})$ acts trivially on the space of modifications $\mathcal{E}_{b} \dashrightarrow \mathcal{E}_{0}$ with its $J_{b}(\mathbb{Q}_{p}) \times G(\mathbb{Q}_{p})$-action by automorphisms of bundles. Moreover, we know both sides of (1) take values in admissible representations of finite length, therefore we obtain a Bernstein decomposition of both sides:
\[ \bigoplus_{\mathfrak{s} \in I_{\chi}(J_{b})} \mathbb{D}_{J_{b}}(R\Gamma_{c}(G,b,\mu)[\pi]_{\mathfrak{s}}) \simeq \bigoplus_{\mathfrak{s} \in I_{\chi^{-1}}(J_{b})} R\Gamma_{c}(G,b,\mu)[\pi^{*}]_{\mathfrak{s}}[-r]  \]
Then, by Proposition 3.1, we see that, for a fixed $\mathfrak{s} = (M,\sigma) \in I_{\chi}(J_{b})$ with $\mathfrak{s}^{*} = (M,\sigma^{*}) \in I_{\chi^{-1}}(J_{b})$, this induces an isomorphism
\[ \mathbb{D}_{J_{b}}(R\Gamma_{c}(G,b,\mu)[\pi]_{\mathfrak{s}}) \simeq R\Gamma_{c}(G,b,\mu)[\pi^{*}]_{\mathfrak{s}^{*}}[-r] \]
of $J_{b}(\mathbb{Q}_{p}) \times W_{E}$-modules. Similarly, we deduce such an isomorphism for the $\rho$-isotypic part. We record this as a Corollary now.
\begin{corollary}
For a basic local Shimura datum $(G,b,\mu)$ such that $B(G,\mu)$ is Hodge-Newton reducible and $\pi$ (resp. $\rho$) a supercuspidal representation of $G(\mathbb{Q}_{p})$ (resp. $J_{b}(\mathbb{Q}_{p})$) with central character $\chi$, we have, for all $\mathfrak{s} = (M,\sigma) \in I_{\chi}(J_{b})$ (resp. $I_{\chi}(G)$) with $\mathfrak{s}^{*} = (M,\sigma^{*}) \in I_{\chi^{-1}}(J_{b})$ (resp. $I_{\chi^{-1}}(G)$), isomorphisms
\[ \mathbb{D}_{J_{b}}(R\Gamma_{c}(G,b,\mu)[\pi]_{\mathfrak{s}}) \simeq R\Gamma_{c}(G,b,\mu)[\pi^{*}]_{\mathfrak{s}^{*}}[-r]\]
and 
\[ \mathbb{D}_{G}(R\Gamma_{c}(G,b,\mu)[\rho]_{\mathfrak{s}}) \simeq R\Gamma_{c}(G,b,\mu)[\rho^{*}]_{\mathfrak{s}^{*}}[-r]\]
of $J_{b}(\mathbb{Q}_{p}) \times W_{E}$ and $G(\mathbb{Q}_{p}) \times W_{E}$-modules, respectively. 
\end{corollary}
Now, with this in hand, we can deduce Theorem 1.2 by passing to cohomology.
\begin{theorem}
For a basic local Shimura datum $(G,b,\mu)$ such that $B(G,\mu)$ is Hodge-Newton reducible, let $\pi$ a supercuspidal representation of $G(\mathbb{Q}_{p})$ with central character $\chi$, $\mathfrak{s} \in I_{\chi}(J_{b})$ with associated integer $i_{J_{b}}(\mathfrak{s})$, $\rho$ a smooth irreducible representation of $J_{b}(\mathbb{Q}_{p})$ with supercuspidal support given by $\mathfrak{s}$, and $\sigma$ an irreducible representation of $W_{E}$. Then $\rho\boxtimes\sigma$ occurs as a subquotient of $H^q(R\Gamma_c(G,b,\mu)[\pi])$ if and only if
$\mathrm{Zel}(\rho)\boxtimes\sigma^*$ occurs as a subquotient of
$H^{-q+i_{J_b}(\mathfrak{s})}(R\Gamma_c(G,b,\mu)[\pi^*])$. The analogous statement is true with the roles
of $\pi$ and $\rho$ interchanged (with $\mathfrak{s}\in I_\chi(G)$ and $i_G(\mathfrak{s})$).

\end{theorem}
\begin{proof}
We explain the proof for $\pi$ supercuspidal, the proof for $\rho$ supercuspidal being the same.
We apply Corollary~3.4 to the contragredient $\pi^*$ and the Bernstein component $\mathfrak{s}^*$;
this gives an identification
\[\mathbb{D}_{J_b}(R\Gamma_c(G,b,\mu)[\pi^*]_{\mathfrak{s}^*})\simeq R\Gamma_c(G,b,\mu)[\pi]_{\mathfrak{s}}[-r].\]
Hence $\rho\boxtimes\sigma$ occurs as a subquotient of $H^q(R\Gamma_c(G,b,\mu)[\pi])$ if and only if it
occurs as a subquotient of $H^{q+r}(\mathbb{D}_{J_b}(R\Gamma_c(G,b,\mu)[\pi^*]_{\mathfrak{s}^*}))$. Since
$\mathbb{D}_{J_b}$ is contravariant, with $W_E$ acting through the contragredient, it follows by
Proposition~3.1 (and the equivalence $\mathrm{Zel} \circ \mathrm{Zel}\simeq\mathrm{id}$) that this occurs if and only if
$\mathrm{Zel}(\rho)\boxtimes\sigma^*$ occurs as a subquotient of
$H^{-q+i_{J_b}(\mathfrak{s})}(R\Gamma_c(G,b,\mu)[\pi^*])$, as desired.
\end{proof}
\printbibliography 
\end{document}